\documentclass[12pt]{article}
\usepackage{graphicx} 
\usepackage{amsmath} 
\usepackage{amssymb} 
\usepackage{amsfonts} 
\usepackage{amsthm} 
\usepackage{longtable} 
\usepackage{comment}

\usepackage{color}
\usepackage{subfig}
\input{epsf}
\input xy
\xyoption{all}
\usepackage{latexsym}
\newtheorem{lemma}{Lemma}[section] 
\newtheorem{theorem}[lemma]{Theorem}

\renewcommand{\epsilon}{\varepsilon} 
\renewcommand{\bar}{\overline} 
\renewcommand{\hat}{\widehat} 

\renewcommand{\leq}{\leqslant} 
\renewcommand{\geq}{\geqslant}

\newcommand{\Po}{\mathcal{P}}
\newcommand{\LL}{\mathcal{L}}

\newcommand{\F}{\mathcal{F}}

\newcommand{\K}{\mathcal{K}}

\begin{document}


\title{Cube-Like Polytopes and Complexes}
 
\author{Andrew Duke\\
Northeastern University \\
Boston, Massachussetts,  USA, 02115
\and and\\[.05in]
Egon Schulte\thanks{Supported by NSF-grant DMS--0856675}\\
Northeastern University\\
Boston, Massachussetts,  USA, 02115}

\date{ \today }
\maketitle

\begin{abstract}
\noindent
The main purpose of this paper is to popularize Danzer's power complex construction and establish some new results about covering maps between two power complexes. Power complexes are cube-like combinatorial structures that share many structural properties with higher-dimensional cubes and cubical tessellations on manifolds. Power complexes that are also abstract polytopes have repeatedly appeared somewhat unexpectedly in various contexts, although often under a different name. However, the non-polytope case is largely unexplored.
\bigskip\medskip



\end{abstract}

\section{Introduction}

Combinatorial structures built from cubes or cube-like elements have attracted a lot of attention in geometry, topology, and combinatorics. In this paper we study a particularly interesting class of cube-like structures known as {\em power complexes\/}. These power complexes were first discovery by Danzer in the early 1980's (see \cite{dan,arp,esext}). Power complexes that are also abstract polytopes have repeatedly appeared somewhat unexpectedly in various contexts, although often under a different name; for example, see Coxeter~\cite{crsp}, Effenberger-K\"uhnel~\cite{eff}, K\"uhnel~\cite{kuet}, McMullen-Schulte~\cite[Ch. 8]{arp} and Ringel~\cite{ri}. However, most power complexes are not abstract polytopes, and have not been very well researched.

The main purpose of this paper is to popularize Danzer's power complex construction and establish some new results about covering maps between power complexes. Our discussion is in terms of incidence complexes, a class of ranked incidence structures closely related to polytopes, ranked partially ordered sets, and incidence geometries (Danzer-Schulte~\cite{kom1,kom2}). In Section~\ref{bano} we begin by reviewing key facts about incidence complexes and their automorphism groups. Then in Sections~\ref{pow} we define power complexes and establish some of their basic properties. A number of applications of power complexes are summarized in Section~\ref{app}. Finally, Section~\ref{quocov} describes fairly general circumstances that guarantee the existence of covering maps between two power complexes.

\section{Incidence complexes}
\label{bano}

An incidence complex has some of the key combinatorial properties of the face lattice of a convex polytope; in general, however, an incidence complex need not be a lattice, need not be finite, need not be an abstract polytope, and need not admit any familiar geometric realization. The notion of an {\em incidence complex\/} is originally due to Danzer~\cite{dan,kom1} and was inspired by Gr\"unbaum~\cite{grgcd}. Incidence complexes can also be viewed as incidence geometries or diagram geometries with a linear diagram (see Buekenhout-Pasini~\cite{bup}, Leemans~\cite{leem}, Tits~\cite{tib}), although here we study them from the somewhat different discrete geometric and combinatorial perspective of polytopes and ranked partially ordered sets.

Following Danzer-Schulte~\cite{kom1} (and \cite{kom2}), an {\em incidence complex $\K$ of rank $k$\/}, or briefly a {\em $k$-complex\/}, is defined by the properties \textbf{(I1)},\ldots,\textbf{(I4)} below. The elements of $\K$ are called {\em faces\/} of $\K$.\medskip

\noindent\textbf{(I1)}\  $\K$ is a partially ordered set with a unique least face and a unique greatest face.
\medskip

\noindent\textbf{(I2)}\  Every totally ordered subset of $\K$ is contained in a (maximal) totally ordered subset with exactly $k+2$ elements, a {\em flag\/}, of $\K$.
\medskip

These two conditions make $\K$ into a ranked partially ordered set, with a strictly monotone rank function with range 
$\{-1,0,\ldots,k\}$. A face of rank $i$ is called an $i$-\textit{face}; often $F_i$ will indicate an $i$-face. The least face and greatest face are the {\em improper faces\/} of $\K$ and have ranks $-1$ and $k$, respectively; all other faces of $\K$ are {\em proper faces\/} of $\K$. A face of rank $0$, $1$ or $n-1$ is also called a {\em vertex\/}, an {\em edge\/} or a {\em facet}, respectively. We let $\F(\K)$ denote the set of flags of $\K$.
\medskip

\noindent\textbf{(I3)}\  $\K$ is {\em strongly flag-connected\/}, meaning that if $\Phi$ and $\Psi$ are two flags of $\K$, then there is a finite sequence of flags $\Phi=\Phi_0,\Phi_1,\ldots,\Phi_{m-1},\Phi_{m}=\Psi$, all containing $\Phi\cap\Psi$, such that successive flags are {\em adjacent\/} (differ in just one face).
\medskip

Call two flags {\em $i$-adjacent\/}, for $i=0,\ldots,k-1$, if they are adjacent and differ exactly in their $i$-faces. With this notion of adjacency, $\F(\K)$ becomes the \textit{flag graph} for $\K$ and acquires a natural edge-labelling where edges labelled $i$ represent pairs of $i$-adjacent flags. 

Our last defining condition is a homogeneity requirement for the numbers of $i$-adjacent flags for each $i$.
\medskip

\noindent\textbf{(I4)}\  There exist cardinal numbers $c_0,\ldots,c_{k-1}\geq 2$, for our purposes taken to be finite, such that, whenever $F$ is an $(i-1)$-face and $G$ a $(i+1)$-face with $F < G$, the number of $i$-faces $H$ with $F<H<G$ equals $c_i$.
\medskip

If $F$ is an $i$-face and $G$ a $j$-face with $F < G$, we call
\[ G/F := \{ H \in \K \, | \, F \leq H \leq G \} \]
a \textit{section} of $\K$. It follows that $G/F$ is an incidence complex in its own right, of rank $j-i-1$ and with cardinal numbers $c_{i+1},\ldots,c_{j-1}$. It is useful to identify a $j$-face $G$ of $\K$ with the $j$-complex $G/F_{-1}$. Likewise, if $F$ is an $i$-face, the $(k-i-1)$-complex $F_k/F$ is the {\em co-face\/} of $F$ in $\K$; if $F$ is a vertex (and $i=0$), this is also called the \textit{vertex-figure} at $F$.

An {\em abstract $k$-polytope\/}, or simply {\em $k$-polytope\/}, is an incidence complex of rank $k$ such that $c_i=2$ for $i=0,\ldots,k-1$ (see McMullen-Schulte~\cite{arp}). Thus a polytope is a complex in which every flag has precisely one $i$-adjacent flag for each $i$. For polytopes, the last condition (I4) is also known as the {\em diamond condition\/}. 

The \textit{automorphism group} $\Gamma(\K)$ of an incidence complex $\K$ consists of all order-preserving bijections of $\K$. We say that $\K$ is \textit{regular} if $\Gamma(\K)$ is transitive on the flags of $\K$. Note that a regular complex need not have a simply flag-transitive automorphism group (in fact, $\Gamma(\K)$ may not even have a simply flag-transitive subgroup), so in general $\Gamma(\K)$ has nontrivial flag-stabilizer subgroups. However, the group of a regular polytope is always simply flag-transitive.

It was shown in \cite{kom2} (for a proof for polytopes see also \cite[Ch. 2]{arp}) that the group $\Gamma:=\Gamma(\K)$ of a regular $k$-complex $\K$ has a well-behaved system of generating subgroups. Let $\Phi:=\{F_{-1},F_0,\ldots,F_k\}$ be a fixed, or {\em base flag\/}, of $\K$, where $F_i$ designates the $i$-face in $\Phi$ for each $i$. For each $\Omega\subseteq \Phi$ let $\Gamma_\Omega$ denote the stabilizer of $\Omega$ in $\Gamma$. Then $\Gamma_\Phi$ is the stabilizer of the base flag $\Phi$, and $\Gamma_\emptyset = \Gamma$. Moreover, for $i=-1,0,\ldots,k$ set 
\[ R_{i} :=  \Gamma_{\Phi\setminus\{F_i\}} = \langle \varphi\in \Gamma \mid F_j\varphi =F_j \mbox{ for all } j\neq i\rangle .\]
Then each $R_i$ contains $\Gamma_\Phi$, and coincides with $\Gamma_\Phi$ when $i=-1$ or $k$; in particular,
\begin{equation}
\label{ci}
c_i := |R_{i}:\Gamma_\Phi | \quad\; (i=0,\ldots,k-1). 
\end{equation}
Moreover, these subgroups have the following commutation property:
\begin{equation}
\label{commu}
R_i \cdot R_j = R_j \cdot R_i \qquad (-1\leq i < j-1 \leq k-1). 
\end{equation}
Note here that $R_i$ and $R_j$ commute as subgroups, not generally at the level of elements.

The groups $R_{-1},R_0,\ldots,R_{k}$ form a {\em distinguished system of generating subgroups\/} of $\Gamma$, that is,
\begin{equation}
\label{genga}
\Gamma = \langle R_{-1},R_0,\ldots,R_{k} \rangle .
\end{equation}
Here the subgroups $R_{-1}$ and $R_{k}$ are redundant when $k>0$.  More generally, if $\Omega$ is a proper subset of $\Phi$, then 
\[ \Gamma_\Omega = \langle R_i \mid -1\leq i\leq k,\,F_{i}\not\in\Omega\rangle .\]

For each nonempty subset $I$ of $\{-1,0,\ldots,k\}$ define $\Gamma_{I} := \langle R_i \mid i\in I\rangle$; and for $I=\emptyset$ define $\Gamma_I := R_{-1}=\Gamma_\Phi$. (As a warning, the notation $\Gamma_\emptyset$ can have two meanings, either as $\Gamma_\Omega$ with $\Omega=\emptyset$ or $\Gamma_I$ with $I=\emptyset$; the context should make it clear which of the two is being used.)  Thus 
\[ \Gamma_{I} = \Gamma_{\{F_j\mid j\not\in I\}} \quad (I\subseteq\{-1,0,\ldots,k\}); \] 
or equivalently, 
\[ \Gamma_\Omega = \Gamma_{\{i\mid F_i \not\in\Omega\}} \quad (\Omega\subseteq \Phi). \]

The automorphism group $\Gamma$ of $\K$ and its distinguished generating system satisfy the following important {\em intersection property\/}: 
\begin{equation}
\label{intprop}
\Gamma_I \cap \Gamma_J = \Gamma_{I\cap J}\qquad (I,J\subseteq \{-1,0,\ldots,k\}) . 
\end{equation}

The combinatorial structure of $\K$ can be completely described in terms of the distinguished generating system of $\Gamma(\K)$. In fact, bearing in mind that $\Gamma$ acts transitively on the faces of each rank, the partial order is given by
\[ F_{i}\varphi \leq F_{j}\psi\; \longleftrightarrow\; \psi^{-1}\varphi \in \Gamma_{\{i+1,\ldots,k\}}\Gamma_{\{-1,0,\ldots,j-1\}}
\;\quad (-1\leq i\leq j\leq k;\, \varphi,\psi\in\Gamma) ,\]
or equivalently,
\begin{equation}
\label{partorder}
F_{i}\varphi \leq F_{j}\psi\, \longleftrightarrow \,
\Gamma_{\{-1,0,\ldots,k\}\setminus\{i\}}\varphi \cap \Gamma_{\{-1,0,\ldots,k\}\setminus\{j\}}\psi \neq \emptyset
\quad (-1\leq i\leq j\leq k;\, \varphi,\psi\in\Gamma). 
\end{equation}

Conversely, if $\Gamma$ is any group with a system of subgroups $R_{-1},R_0,\ldots,R_k$ such that (\ref{commu}), (\ref{genga}) and (\ref{intprop}) hold, and $R_{-1}=R_k$, then $\Gamma$ is a flag-transitive subgroup of the full automorphism group of a regular incidence complex $\K$ of rank $k$ (see again \cite{kom2}, or \cite[Ch. 2]{arp} for polytopes). The $i$-faces of $\K$ are the right cosets of $\Gamma_{\{-1,0,\ldots,k\}\setminus\{i\}}$ for each $i$, and the partial order is given by (\ref{partorder}). The homogeneity parameters $c_0,\ldots,c_{k-1}$ are determined by (\ref{ci}).

For abstract regular polytopes, these structure results lie at the heart of much research activity in this area (see \cite{arp}). In this case the flag stabilizer $\Gamma_\Phi$ is the trivial group, and each nontrivial subgroup $R_i$ (with $i\neq -1,k$) has order $2$ and is generated by an involutory automorphism $\rho_i$ that maps $\Phi$ to its unique $i$-adjacent flag. The group of an abstract regular polytope is then a {\em string C-groups\/}, meaning that the {\em distinguished involutory generators\/} $\rho_0, \ldots, \rho_{k-1}$ satisfy both the commutativity relations typical of a Coxeter group with string diagram, and the intersection property~(\ref{intprop}).

\section{Power complexes}
\label{pow}

In this section we briefly review the construction of the {\em power complexes\/} $n^\K$, an interesting family of incidence complexes with $n$ vertices on each edge, and with each vertex-figure isomorphic to $\K$ (see \cite{esext}, and \cite[Section 8D]{arp} for $n=2$). These power complexes were first discovered by Danzer in the early 1980's; however, the construction announced in~\cite{dan} was never published by Danzer, and first appeared in print in~\cite{esext}. The power complexes $n^\K$, with $n=2$ and $\K$ a polytope, are abstract polytopes and have attracted a lot of attention (see \cite[Ch. 8]{arp}). In a sense, these power complexes are generalized cubes; and in certain cases (when $\K$ has simplex facets) they can also be viewed as cubical complexes (see \cite{buch,nov}). We briefly review some applications in Section~\ref{app}.

To begin with, we say that an (incidence) complex $\K$ is {\em vertex-describable} if its faces are uniquely determined by their vertex-sets. A complex is vertex-describable if and only if its underlying face poset can be represented by a family of subsets of the vertex-set ordered by inclusion. If a complex $\K$ is a lattice, then $\K$ is vertex-describable. For example, the torus map $\K=\{4,4\}_{(s,0)}$ is vertex-describable if and only if $s\geq 3$. The faces of a vertex-describable complex are again vertex-describable.

Now let $n\geq 2$ and define $N:=\{1,\ldots,n\}$. Suppose $\K$ is a finite vertex-describable $(k-1)$-complex with $v$ vertices and vertex-set $V:=\{1,\ldots,v\}$. Then $\Po:=n^\K$ will be a finite $k$-complex with vertex-set 
\begin{equation}
\label{nv}
N^v = \bigotimes_{i=1}^{v} N ,
\end{equation}
the cartesian product of $v$ copies of $N$; its $n^v$ vertices are written as row vectors 
$\varepsilon := (\varepsilon_1,\ldots,\varepsilon_v)$. Now recall that, since $\K$ is vertex-describable, we may view the faces of $\K$ as subsets of $V$. With this in mind we take as $j$-faces of $\Po$, for any $(j-1)$-face $F$ of $\K$ and any vector $\varepsilon = (\varepsilon_1,\ldots,\varepsilon_v)$ in $N^v$, the subsets $F(\varepsilon)$ of $N^v$ defined by
\begin{equation}
\label{fep}
F(\varepsilon) := \{(\eta_1,\ldots,\eta_{v})\in N^v \! \mid \eta_{i} = \varepsilon_{i} \mbox{ if } i\not\in F\}
\end{equation}
or, abusing notation, the cartesian product
\[ F(\varepsilon) \,:=\, ( \bigotimes_{i \in F} N ) \times ( \bigotimes_{i \not\in F} \{\varepsilon_i\} ). \]
In other words, the $j$-face $F(\varepsilon)$ of $\mathcal{P}$ consists of the vectors in $N^v$ that coincide with $\varepsilon$ precisely in the components determined by the vertices of $\K$ not lying in the $(j-1)$-face $F$ of $\K$. It follows that, if $F$, $F'$ are faces of $\K$ and $\varepsilon=(\varepsilon_1,\ldots,\varepsilon_v)$, $\varepsilon' =(\varepsilon_1,\ldots,\varepsilon_v)$ are vectors in~$N^v$, then $F(\varepsilon) \subseteq F'(\varepsilon')$ if and only if $F \leq F'$ in $\K$ and $\varepsilon_i = \varepsilon_{i}'$ for each $i$ not contained in~$F'$. 

It can be shown that the set of all faces $F(\varepsilon)$, where $F$ is a face of $\K$ and $\varepsilon$ a vector in $N^v$, partially ordered by inclusion (and supplemented by the empty set as least face), is an incidence complex of rank $k$. This is the desired complex $\Po=n^\K$. 

The following theorem summarizes a number of key properties of power complexes.

\begin{theorem}
\label{propnk}
Let $\K$ be a finite incidence complex of rank $k-1$ with $v$ vertices, and let $\K$ be vertex-describable. Then the power complex $\Po:=n^\K$ has the following properties.\\[.04in]
(a) $\Po$ is an incidence complex of rank $k$ with vertex-set $N^v$ and each vertex-figure isomorphic to $\K$.\\[.02in]
(b) If $F$ is a $(j-1)$-face of $\K$ and $\F:=F/F_{-1}$ is the $(j-1)$-complex determined by~$F$, then the $j$-faces of $\Po$ of the form $F(\varepsilon)$ with $\varepsilon$ in $N^v$ are isomorphic to the power complex $n^\F$ of rank~$j$.\\[.02in]
(c) $\Gamma(\Po)$ contains a subgroup $\Lambda$ isomorphic to $S_n\wr \Gamma(\K) = S_{n}^{v}\rtimes \Gamma(\K)$, the wreath product of $S_n$ and $\Gamma(\mathcal{K})$ defined by the natural action of $\Gamma(\K)$ on the vertex-set of $\K$. Moreover, $\Lambda$ acts vertex-transitively on $\Po$ and has vertex stabilizers isomorphic to $S_{n-1}\wr \Gamma(\K)$.\\[.02in]
(d) If $\mathcal{K}$ is regular, then so is $\Po$. In this case the subgroup $\Lambda$ of $\Gamma(\Po)$ of part~(d) acts flag-transitively on $\Po$; in particular, if $n=2$ and $\K$ is polytope, then $\Lambda=\Gamma(\Po)$.
\end{theorem}

\begin{proof}
For power complexes $2^\K$ regular polytopes $\mathcal{K}$ these facts are well-known (see \cite[Section 8D]{arp} and \cite{psw,esext}). Here we briefly outline the proof for general power complexes, as no general proof has been published anywhere. So, as before, let $\K$ be a finite vertex-describable complex of rank $k-1$.

Begin by making the following important observation regarding inclusion of faces in~$\Po$:\ if $F(\varepsilon) \subseteq F'(\varepsilon')$, with $F, F',\varepsilon,\varepsilon'$ as above, then  $F'(\varepsilon') = F'(\varepsilon)$. Thus, in designating the larger face we may take $\varepsilon'=\varepsilon$.  It follows that every face containing a given vertex $\varepsilon$ must necessarily be of the form $F(\varepsilon)$ with $F\in\mathcal{K}$, and that any two such faces $F(\varepsilon)$ and $F'(\varepsilon)$ are incident in $\Po$ if and only if $F$ and $F'$ are incident in $\K$. As an immediate consequence, $\Po$ must have vertex-figures isomorphic to $\K$. It is straightforward to prove that $\Po$ actually is an incidence complex of rank $k$.

For part (b), let $F$ be a $(j-1)$-face of $\K$ with $v_F$ vertices and vertex-set $V_F$, and let $\varepsilon$ be a vector in $N^v$. Now, if $F'(\varepsilon')$ is any face of $\Po$ with $F'(\varepsilon')\subseteq F(\varepsilon)$ in $\Po$, then necessarily $F'\leq F$ in $\K$ and $\varepsilon_{i}' = \varepsilon_{i}$ for each $i\not\in F$; in other words, the vectors $\varepsilon$ and $\varepsilon'$ agree on each component representing a vertex $i$ of $\K$ that lies outside $F$. It follows that the components of vectors in $N^v$ corresponding to vertices $i$ of $\K$ outside of $F$ do not matter in determining the structure of the $j$-face $F(\varepsilon)$ of $\Po$. Hence, if we omit these components and simply write $\eta_{F}:=(\eta_i)_{i\in F}$ for the ``trace" of a vector $\eta$ on $F$, then $\eta_F$ lies in the cartesian product $N^{v_F}:=\bigotimes_{i\in V_F} N$, and the faces $F(\varepsilon)$ and $F'(\varepsilon')$ of $\Po$ can safely be designated by $F(\varepsilon_F)$ and $F'(\varepsilon_{F}')$, respectively. Then, in particular, $F(\varepsilon_F) = N^{v_F}$ is the unique greatest face of $n^\F$, and $F'(\varepsilon_{F}')$ becomes a face of $n^\F$. Moreover, the partial order on the $j$-face $F(\varepsilon)$ of $\Po$ is just the standard inclusion of faces in $n^\F$. Thus, as a complex, $F(\varepsilon)$ is isomorphic to $n^\F$. This proves part (b).

The automorphism group $\Gamma(\Po)$ always contains a subgroup $\Sigma$ isomorphic to $S_n^v$, the direct product of $v$ copies of the symmetric group $S_n$. In fact, for each $i=1,\ldots,v$, the symmetric group $S_n$ can be viewed as acting on the $i^{th}$ component of the vectors in $N^v$ (while leaving all other components unchanged), and this action on the vertex-set $N^v$ induces an action as a group of automorphisms on $\Po$. In particular, $\Sigma$ acts vertex-transitively on $\Po$, so the same holds for $\Gamma(\Po)$ as well.

Moreover, $\Gamma(\K)$ is naturally embedded in $\Gamma(\Po)$ as a subgroup of the vertex-stabilizer of $\varepsilon=(0,\ldots,0)$ in $\Gamma(\Po)$. In fact, each automorphism $\varphi$ of $\K$ determines an automorphism $\widehat{\varphi}$ of $\Po$ as follows. Define $\widehat{\varphi}$ on the set of vertices $\eta=(\eta_1,\ldots,\eta_v)$ by~\footnote{Throughout we write maps on the right.}
\[(\eta)\widehat{\varphi} := (\eta_{(1)\varphi},\ldots,\eta_{(v)\varphi}) =:\eta_\varphi, \]
and more generally on the set of faces $F(\eta)$ of $\Po$ by
\[ F(\eta)\widehat{\varphi} := (F\varphi)(\eta_\varphi) .\]
Then it is straightforward to verify that $\widehat{\varphi}$ is indeed an automorphism of $\Po$, and that $\widehat{\varphi}$ fixes $\varepsilon=(0,\ldots,0)$. It follows that the two subgroups $\Sigma$ and $\Gamma(\K)$ together generate a subgroup $\Lambda$ of $\Gamma(\Po)$ isomorphic to $S_{n}\wr \Gamma(\K) \cong S_{n}^{v}\rtimes \Gamma(\K)$. Clearly, $\Lambda$ acts vertex-transitively and has vertex-stabilizers isomorphic to $S_{n-1}\wr \Gamma(\K)$. Now part (c) follows. 

Finally, suppose $\K$ is regular. Then $\Lambda$ acts flag-transitively on $\Po$, and so does $\Gamma(\Po)$. Thus $\Po$ is regular. If $n=2$ and $\K$ is a regular polytope, then $\Po$ is also a regular polytope and $\Gamma(\Po)=\Lambda$. This proves part~(d).
\end{proof}

We do not know of an example of a power complex $n^\K$, with $\K$ regular, where the full automorphism group of $n^\K$ is strictly larger than its subgroup $S_n\wr\Gamma(\K)$.

\begin{figure}[h]
\begin{center}
\includegraphics[scale = .5]{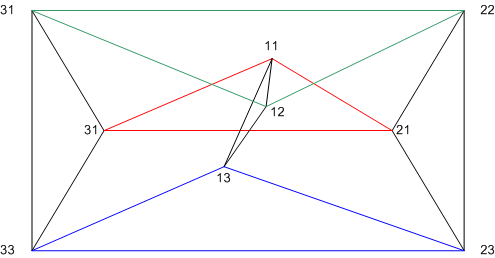}
\caption{Combinatorics of the complex square $\gamma_{2}^{3}$}
\label{fig:compsquare}
\end{center}
\end{figure}

In Section~\ref{app}, we discuss a number of interesting applications of the $n^\K$ construction. Here we just describe the most basic example obtained when $\K=\alpha_{v-1}$ (see \cite{crp}), the $(v-1)$-simplex (with $v$ vertices).  In this case $n^\K$ is combinatorially isomorphic to the {\it complex $v$-cube\/}
\[ \gamma_{v}^{n} =  n \{4\} 2 \{3\} 2 \cdots 2 \{3\} 2  \]
in $v$-dimensional unitary complex $v$-space $\mathbb{C}^v$, that is, $n^{\alpha_{v-1}}= \gamma_{v}^{n}$. The unitary complex symmetry group of $\gamma_{v}^{n}$ is isomorphic to $C_{n}\wr S_v$ (see Coxeter~\cite{cox1} and Shephard~\cite{shep}). However,  the combinatorial automorphism group of $\gamma_{v}^{n}$ is much larger when $n>2$, and includes a subgroup isomorphic to $S_{n}\wr S_v$. The case $n=2$ always gives the ordinary real $v$-cube $\gamma_v:=\gamma_v^2 = \{4,3^{v-2}\}$ (see~\cite{crp}). 

The combinatorics of the {\em complex square\/} $\gamma_{2}^{3}=3\{4\}2$ in $\mathbb{C}^2$ (obtained when $v=2$ and $n=3$) is illustrated in Figure~\ref{fig:compsquare}; there are $9$ vertices (denoted $i\,j$ with $i,j=1,2,3$), each contained in $2$ edges (drawn as $3$-cycles), as well as $6$ edges, each containing $3$ vertices. 

Now let $\K$ be an arbitrary incidence complex of rank $k$, and let $0\leq j\leq k-1$. The {\em $j$-skeleton\/} $skel_j(\K)$ of $\K$ is the incidence complex, of rank $j+1$, whose faces of rank less than or equal to $j$ are those of $\K$, with the partial order inherited from $\K$; as greatest face, of rank $j+1$, we may simply take the greatest face of $\K$. 

The following lemma says that taking skeletons and taking power complexes are commuting operations.  

\begin{lemma}
\label{skel}
Let $\K$ be a finite vertex-describable $k$-complex, let $0\leq j\leq k-1$, and let $n\geq 2$. Then
\[ skel_{j+1}(n^\K) = n^{skel_j(\K)} .\]
\end{lemma}

\begin{proof}
The proof is straightforward. First note that a skeleton of a vertex-describable complex is again vertex-describable, with the same vertex set as the underlying complex. The proper faces of $skel_{j+1}(n^\K)$ are the faces $F(\varepsilon)$ of $n^\K$ where $F$ has rank at most $j$ and $\varepsilon$ lies in $N^v$. On the other hand, the proper faces of $n^{skel_j(\K)}$ are of the form $F(\varepsilon)$ where $F$ is a face of $skel_j(\K)$ of rank at most $j$ and $\varepsilon$ lies in $N^v$. But the faces of $\K$ of rank at most $j$ are precisely the faces of $skel_j(\K)$ of rank at most $j$. Now the lemma follows. 
\end{proof}

We conclude this section with a nice application of the lemma. Suppose $n\geq 2$ and $\K$ is the (unique) complex of rank $1$ with $v$ vertices. Now identifying $\K$ with $skel_{0}(\alpha_{v-1})$ we then have
\begin{equation}
\label{oneskel}n^\K = n^{skel_{0}(\alpha_{v-1})} = skel_{1}(n^{\alpha_{v-1}}) = skel_{1}(\gamma_v^n).
\end{equation}
Thus the $2$-complex $n^\K$ is isomorphic to the $1$-skeleton of the unitary complex $v$-cube $\gamma_v^n$ described above. 

\section{Applications}
\label{app}

In this section we briefly review a number of interesting applications of the power complex construction that have appeared in the literature.

First suppose $n=2$ and $\mathcal{K}=\{q\}$ is a $q$-gon with $3\leq q<\infty$. It was shown in \cite[Ch. 8D]{arp} that $2^{\{q\}}$ is isomorphic to Coxeter's regular map $\{4,q \!\mid\!4^{\lfloor q/2\rfloor -1}\}$ in the $2$-skeleton of the ordinary $q$-cube $\gamma_{q}=\{4,3^{q-2}\}$, whose edge-graph coincides with that of the cube (see \cite[p. 57]{crsp}). In fact, the method of construction directly produces a realization of $2^{\{q\}}$ in the $2$-skeleton of $\gamma_{q}$, which is identical with the realization outlined in~\cite{crsp}. This map and its realizations were rediscovered several times in the literature. For example, Ringel~\cite{ri} and Beineke-Harary~\cite{beh} established that the genus $2^{q-3}(q-4)+1$ of Coxeter's map is the smallest genus of any orientable surface into which the edge-graph of the $q$-cube can be embedded without self-intersections. It is rather surprising that each map $\{4,q \!\mid\! 4^{\lfloor q/2\rfloor -1}\}$, as well as its dual $\{q,4 \!\mid\! 4^{\lfloor q/2\rfloor -1}\}$, can also be embedded as a polyhedron without self-intersections in ordinary $3$-space (see McMullen-Schulz-Wills~\cite{msw2} and McMullen-Schulte-Wills~\cite{msw}). When $q \geq 12$, the genus of this polyhedron exceeds the number of vertices, $2^q$, of $\{4,q \!\mid\! 4^{[q/2]-1}\}$, which is somewhat hard to visualize.

When $n=2$ and $\mathcal{K}$ is an abstract $2m$-polytope given by a neighborly simplicial $(2m-1)$-sphere, the corresponding power complex $2^\mathcal{K}$ gives an {\em $m$-Hamiltonian\/} $2m$-manifold embedded as a subcomplex of a higher-dimensional cube (see K\"uhnel-Schulz~\cite{kusc}, Effenberger-K\"uhnel~\cite{eff}). Recall here that a polytope is {\em neighborly\/} if any two of its vertices are joined by an edge. The $m$-Hamiltonicity then refers to the distinguished property that $2^\mathcal{K}$ contains the full $m$-skeleton of the ambient cube. In this sense, Coxeter's map $\{4,q \!\mid\! 4^{\lfloor q/2\rfloor -1}\}$ gives a $1$-Hamiltonian surface. 

The case when $n=2$ and $\K$ is an (abstract) regular polytope has inspired a number of generalizations of the $2^\K$ construction that have proved important in the study of universality questions and extensions of regular polytopes (see \cite[Ch. 8]{arp}). A particularly versatile generalization is to polytopes $2^{\mathcal{K},\mathcal{D}}$, where $\K$ is a vertex-describable regular $k$-polytope with $v$ vertices, and $\mathcal{D}$ is a Coxeter diagram on $v$ nodes admitting a suitable action of $\Gamma(\mathcal{K})$ as a group of diagram symmetries. The corresponding Coxeter group $W(\mathcal{D})$ then can be extended by $\Gamma(\K)$ to obtain the automorphism group $W(\mathcal{D})\ltimes \Gamma(\mathcal{K})$ of a regular $(k+1)$-polytope denoted $2^{\mathcal{K},\mathcal{D}}$. This polytope is generally infinite, and its vertex-figures are isomorphic to $\K$.  When $\mathcal{D}$ is the {\em trivial\/} diagram, without branches, on the vertex set of $\mathcal{K}$, the $(k+1)$-polytope $2^{\mathcal{K},\mathcal{D}}$ is isomorphic to the power complex $2^\mathcal{K}$ and the Coxeter group $W(\mathcal{D})$ is just $C_2^{v}$. This provides an entirely different construction of power complexes $2^\K$ based on regular polytopes $\K$. 

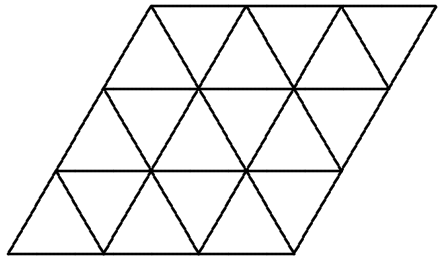
\begin{figure}[h]
\begin{center}
\begin{picture}(158,110)
\put(0,7){
\multiput(0,0)(0.5,0){216}{\circle*{1}}
\multiput(18,31.17)(.5,0){216}{\circle*{1}}
\multiput(36,62.35)(.5,0){216}{\circle*{1}}
\multiput(54,93.55)(.5,0){216}{\circle*{1}}
\multiput(0,0)(.25,.435){216}{\circle*{1}}
\multiput(36,0)(.25,.435){216}{\circle*{1}}
\multiput(72,0)(.25,.435){216}{\circle*{1}}
\multiput(108,0)(.25,.435){216}{\circle*{1}}
\multiput(36,0)(-.25,.435){72}{\circle*{1}}
\multiput(72,0)(-.25,.435){144}{\circle*{1}}
\multiput(108,0)(-.25,.435){216}{\circle*{1}}
\multiput(126,31.175)(-.25,.435){144}{\circle*{1}}
\multiput(144,62.35)(-.25,.435){72}{\circle*{1}}}
\end{picture}
\caption{The torus map $\{3,6\}_{(3,0)}$}
\label{tormap}
\end{center}
\end{figure}

The polytopes $2^{\mathcal{K},\mathcal{D}}$ are very useful in the study of universal regular polytopes, as the following example illustrates (again, see \cite[Ch. 8]{arp}). Let $\mathcal{K}$ be a (polytopal) regular map of type $\{3,r\}$ on a surface, for instance,  a torus map $\{3,6\}_{(b,0)}$ or $\{3,6\}_{(b,b)}$ (see Figure~\ref{tormap}). Suppose we wish to investigate regular $4$-polytopes (if they exist) with cubes $\{4,3\}$ as facets and with copies of $\K$ as vertex-figures. In particular this would involve determining when the universal such structure, denoted 
\[ \mathcal{U}:=\{\{4,3\},\mathcal{K}\},\] 
is a finite polytope. It turns out that this  universal polytope $\mathcal{U}$ always exists (for any $\K$), and that $\mathcal{U}=2^{\mathcal{K},\mathcal{D}}$ for a certain Coxeter diagram $\mathcal{D}$ depending on $\K$ (see \cite[Thm. 8E10]{arp}). In particular, $\mathcal{U}$ is finite if and only if $\mathcal{K}$ is neighborly. In this case $\mathcal{U}=2^\mathcal{K}$ and $\Gamma(\mathcal{U}) = C_2^v \ltimes \Gamma(\mathcal{K})$ (and $\mathcal{D}$ is trivial). For example, if $\K$ is the hemi-icosahedron $\{3,5\}_5$ (with group $[3,5]_5$), then 
\[ \mathcal{U}=\{\{4,3\},\{3,5\}_5\}=2^{\{3,5\}_5}\] 
and 
\[\Gamma(\mathcal{U}) = C_{2}\wr [3,5]_5 = C_{2}^{6}\ltimes [3,5]_5.\] 

\section{Coverings}
\label{quocov}

In this section we investigate coverings of power complexes. We begin with some terminology; see \cite[Ch. 2D]{arp} for similar notions for abstract polytopes.

Let $\K$ and $\LL$ be (incidence) complexes of rank $k$. A map $\gamma: \K\rightarrow \LL$ is called a {\em homomorphism\/} if $\gamma$ preserves incidence in one direction; that is, $F\gamma \leq G\gamma$ in $\LL$ whenever $F\leq G$ in $\K$. (Automorphisms are bijections that are order preserving in both directions.)  A homomorphism $\gamma$ is a \textit{rap-map} if $\gamma$ is \underbar{r}ank preserving and \underbar{a}djacency \underbar{p}reserving; that is, faces of $\K$ are mapped to faces of $\LL$ of the same rank, and pairs of adjacent flags of $\K$ are mapped onto pairs of adjacent flags of $\LL$. A surjective rap-map $\gamma$ is called a \textit{covering} ({\em map\/}). Similarly we call a homomorphism $\gamma: \K\rightarrow \LL$ a {\em weak rap-map\/} if $\gamma$ is rank preserving and {\em weakly adjacency preserving\/}, meaning that $\gamma$ maps a pair of adjacent flags of $\K$ onto a pair of flags of $\LL$ that are adjacent or identical.  

Figure \ref{fig:rap-map1} illustrates an example of a covering $\gamma: \K\rightarrow \LL$ between a hexagon $\K$ with vertices $1,\ldots,6$, and a triangle $\LL$ with vertices $1,2,3$, given by $i, i+3\mapsto i$ for $i=1,2,3$. The edges are mapped by $\{i,i+1\}$, $\{i+3,i+4\} \mapsto \{i,i+1\}$. Thus $\gamma$ wraps the hexagon twice around the triangle.

\begin{figure}[h]
\begin{center}
\includegraphics[scale = .5]{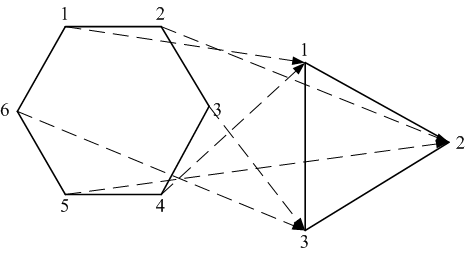}
\caption{A hexagon wrapped around the triangle}
\label{fig:rap-map1}
\end{center}
\end{figure}

Returning to the general discussion, let $\K$ be a $k$-complex and $\Sigma$ be a subgroup of $\Gamma(\K)$. Denote the set of orbits of $\Sigma$ in $\K$ by $\K/\Sigma$, and the orbit of a face $F$ of $\K$ by $F\cdot\Sigma$. Then introduce a partial ordering on $\K/\Sigma$ as follows: if $\widehat{F},\widehat{G}\in\K/\Sigma$, then $\widehat{F}\leq\widehat{G}$ if and only if $\widehat{F}=F\cdot\Sigma$ and $\widehat{G} = G\cdot\Sigma$ for some faces $F$ and $G$ of $\K$ with $F\leq G$. The set $\K/\Sigma$ together with this partial order is the \textit{quotient of $\K$ with respect to }$\Sigma$. The triangle in Figure~\ref{fig:rap-map1} is a quotient of the hexagon obtained by identifying opposite vertices; here $\Sigma$ is generated by the central involution
in $D_6$, the group of the hexagon. 
\eject

\noindent 
{\bf Coverings $n^\K \rightarrow m^\LL$ with $n\geq m$.}
\smallskip

The following theorem says that coverings between (vertex-describable) incidence complexes naturally induce coverings or weak coverings between the corresponding power complexes.

\begin{theorem}
\label{covind}
Let $\K$ and $\LL$ be finite vertex-describable incidence complexes of rank $k$, and let $\gamma: \K \rightarrow \LL$ be a covering. Moreover, let $n\geq m \geq 2$ and $f\!:\!\{1,\ldots,n\}\rightarrow \{1,\ldots,m\}$ be a surjective mapping. Then $\gamma$ and $f$ induce a weak covering $\pi_{\gamma,f}\!:\, n^\K \rightarrow m^\LL$ between the power complexes $n^\K$ and $m^\LL$. Moreover, $\pi_{\gamma,f}$ is a covering if and only if $f$ is a bijection (and $n=m$).
\end{theorem}

\begin{proof}
Suppose $V(\K):=\{1,\ldots,v(\K)\}$ and $V(\LL):=\{1,\ldots,v(\LL)\}$ are the vertex sets of $\K$ and $\LL$, respectively. (It will be clear from the context if a label $j$ refers to a vertex of $\K$ or a vertex of $\LL$.)  Then $v(\LL)\leq v(\K)$ since there is a covering map from $\K$ to $\LL$. Define $N:=\{1,\ldots,n\}$ and $M:=\{1,\ldots,m\}$. 

First note that a typical flag in $n^\K$ has the form 
\[ \Phi(\varepsilon) := \{\emptyset, \varepsilon, F_{0}(\varepsilon),\ldots,F_{k}(\varepsilon)\} , \]
where $\varepsilon$ is a vector in $N^{v(\K)}$ and $\Phi:=\{F_{-1},F_0,\ldots,F_k\}$ is a flag of $\K$. Clearly, if $r\geq 1$ and $\Phi,\Phi'$ are $(r-1)$-adjacent flags of $\K$, then $\Phi(\varepsilon),\Phi'(\varepsilon)$ are $r$-adjacent flags of $n^\K$. Similar statements also hold for $m^\LL$.

Now consider the given covering map $\gamma: \K \rightarrow \LL$. For a vertex $j$ of $\K$ write $\bar{j}:=j\gamma$, so $\bar{j}$ is a vertex of $\LL$. Since $\gamma$ is surjective, we may assume that the vertex labeling for $\K$ and $\LL$ is such that the vertices $\bar{1},\bar{2},\ldots,\bar{v(\LL)}$ comprise all the vertices of $\LL$, and in particular that $\bar{j}=j$ for each $j=1,\ldots,v(\LL)$. Now define the mapping
\begin{equation}
\label{nga}
\begin{array}{rccc}
\pi_{\gamma,f}\!: & n^\K & \rightarrow & m^\LL \\[.02in]
& F(\varepsilon) & \rightarrow & (F\gamma)(\varepsilon_f),
\end{array}
\end{equation}
where as usual $F$ denotes a face of $\K$ and $\varepsilon$ a vector in $N^{v(\K)}$, and  
\[ \varepsilon_f := (\varepsilon_1 f,\ldots,\varepsilon_{v(\LL)} f) \]
is the vector in $M^{v(\LL)}$ given by the images under $f$ of the first $v(\LL)$ components of $\varepsilon$. We claim that $\pi:=\pi_{\gamma,f}$ is a well-defined weak covering.

First we prove that $\pi$ is well-defined. For a face $F$ of a complex we let $V(F)$ denote its vertex set. Now suppose we have $F(\varepsilon)=F'(\varepsilon')$ in $n^\K$, where 
$\varepsilon = (\varepsilon_1,\ldots,\varepsilon_{v(\K)})$ and 
$\varepsilon' = (\varepsilon'_1,\ldots,\varepsilon'_{v(\K)})$ belong to $N^{v(\K)}$ and $F,F'$ are faces of $\K$. Then necessarily $F=F'$, since the vertex sets of $F$ and $F'$ must be the same; recall here that $\K$ is vertex-describable. Thus $F\gamma=F'\gamma$. Moreover, $\varepsilon_{i}=\varepsilon'_{i}$ for each $i\not\in V(F)=V(F')$, so $\varepsilon_f$ and $\varepsilon'_{\!f}$ certainly agree on all components indexed by vertices $i$ with $i\not\in V(F)$. All other components of $\varepsilon_f$ and $\varepsilon'_f$ are indexed by a vertex $i$ of $F$; but if $i\in V(F)$ then $\bar{i} = (i)\gamma\in V(F\gamma) = V(F'\gamma)$, and hence $i$ indexes a component where entries are allowed to range freely over $M=(N)f$. Therefore, 
$(F\gamma)(\varepsilon_f)=(F'\gamma)(\varepsilon'_f)$. Thus $\pi$ is well-defined.

Clearly, $\pi$ is a homomorphism since this is true for $\gamma$. For the same reason, $\pi$ is rank-preserving and surjective. 

It remains to show that $\pi$ is weakly adjacency preserving. To this end, let 
\[ \Phi(\varepsilon) := \{\emptyset, \varepsilon, F_{0}(\varepsilon),\ldots,F_{k}(\varepsilon)\},\;\;\,  
\Phi'(\varepsilon') := \{\emptyset, \varepsilon', F'_{0}(\varepsilon'),\ldots,F'_{k}(\varepsilon')\}\] 
be flags of $n^\K$, where 
\[ \Phi:=\{F_{-1},F_0,\ldots,F_k\},\;\;\, \Phi':=\{F'_{-1},F'_0,\ldots,F'_k\}\] 
are flags of $\K$ and $\varepsilon,\varepsilon'$ are vectors in $N^{v(\K)}$. Suppose $\Phi(\varepsilon)$ and $\Phi'(\varepsilon')$ are $r$-adjacent for some $r\geq 0$. Then two possibilities can arise. 

If $r>0$, then $\varepsilon=\varepsilon'$ and $\Phi,\Phi'$ must be $(r-1)$-adjacent flags of $\K$. It follows that $\varepsilon_f =\varepsilon'_f$, and that $\Phi\gamma,\Phi'\gamma$ are $(r-1)$-adjacent flags of $\LL$ since $\gamma$ is adjacency preserving. Hence the image flags of $\Phi(\varepsilon)$ and $\Phi'(\varepsilon')$ under $\pi$, which are given by
\[ (\Phi(\varepsilon))\pi = 
\{\emptyset, \varepsilon_f,(F_{0}\gamma)(\varepsilon_f),\ldots,(F_{k}\gamma)(\varepsilon_f)\} \]
and 
\[ (\Phi'(\varepsilon'))\pi = 
\{\emptyset, \varepsilon'_f,(F'_{0}\gamma)(\varepsilon'_f),\ldots,(F'_{k}\gamma)(\varepsilon'_f)\} \]
respectively, are also $r$-adjacent. Thus, when $r>0$, the map $\pi$ takes $r$-adjacent flags of $n^\K$ to $r$-adjacent flags of $m^\LL$.

Now suppose $r=0$. Then $\Phi=\Phi'$ (but $\varepsilon\neq\varepsilon'$), since the faces $F_s$ and $F'_{s}$ of $\K$ must have the same vertex sets for each $s\geq 0$; bear in mind that $\K$ is vertex-describable. Moreover, since $F_0=F'_0$ and $r\neq 1$, we have $F_{0}(\varepsilon)=F'_{0}(\varepsilon')=F_{0}(\varepsilon')$, so $\varepsilon_{i}=\varepsilon'_{i}$ for each vertex $i$ of $\K$ distinct from $i_{0}:=F_0$; hence $\varepsilon$ and $\varepsilon'$ differ exactly in the position indexed by $i_0$. Then we certainly have $(F_{s}\gamma)(\varepsilon_f)=(F'_{s}\gamma)(\varepsilon'_f)$ for all $s\geq 0$. Hence $(\Phi(\varepsilon))\pi$ and $(\Phi'(\varepsilon'))\pi$ are either $0$-adjacent or identical. 

At this point we know that $\pi\!: n^\K \rightarrow m^\LL$ is weakly adjacency preserving, that is, $\pi$ is a weak covering. This proves the first part of the theorem. 

Moreover, since $\varepsilon$ and $\varepsilon'$ differ only in the position indexed by $i_0$, the corresponding shortened vectors $(\varepsilon_1,\ldots,\varepsilon_{v(\LL)})$ and $(\varepsilon'_1,\ldots,\varepsilon'_{v(\LL)})$ in $N^{v(\LL)}$ (underlying the definition of $\varepsilon_f$ and $\varepsilon'_{f}$) also differ only in the position indexed by $i_0$; note here that $\bar{i_0}=i_0$, by our labeling of the vertices in $\K$ and $\LL$. Hence the two vertices $\varepsilon_f = (\varepsilon_{1}f,\ldots,\varepsilon_{v(\LL)}f)$ and $\varepsilon'_f = (\varepsilon'_{1}f,\ldots,\varepsilon'_{v(\LL)}f)$ of $m^\LL$ in $(\Phi(\varepsilon))\pi$ and $(\Phi'(\varepsilon'))\pi$, respectively, either coincide or differ in a single position, indexed by $i_{0}$; the former occurs precisely when $\varepsilon_{i_0}f=\varepsilon'_{i_0}f$. Therefore, since $\varepsilon_{i_0}$ and $\varepsilon'_{i_0}$ can take any value in $N$, the mapping $\pi$ is a covering if and only if $f$ is a bijection. This completes the proof.
\end{proof}
\medskip

\noindent 
{\bf Coverings $n^\K \rightarrow m^\LL$ with $n^l\geq m$.}
\smallskip

The previous Theorem~\ref{covind} describes quite general circumstances under which coverings or weak coverings between power complexes $n^\K$ and $m^\LL$ are guaranteed to exist. Under the basic condition that $n\geq m$ this generally leads to a host of possible weak covering maps. Our next theorem deals with coverings or weak coverings between power complexes in situations where certain well-behaved (equifibered) coverings between the original complexes $\K$ and $\LL$ exist. This also permits many examples with $n\leq m$. 

To begin with, let $\K$ and $\LL$ be finite vertex-describable incidence complexes of rank~$k$, and let $V(\K):=\{1,\ldots,v(\K)\}$ and $V(\LL):=\{1,\ldots,v(\LL)\}$, respectively, denote their vertex sets. Suppose there is a covering $\gamma\!: \K \rightarrow \LL$ that is {\em equifibered\/} (with respect to the vertices), meaning that the {\em fibers\/} $\gamma^{-1}(j)$ of the vertices $j$ of $\LL$ under $\gamma$ all have the same cardinality, $l$ (say). In other words, the restriction of $\gamma$ to the vertex sets of $\K$ and $\LL$ is $l:1$, so in particular $v(\K)=l\cdot v(\LL)$.

Important examples of this kind are given by the regular $k$-polytopes $\K$ that are ({\em properly\/}) {\em centrally symmetric\/}, in the sense that the group $\Gamma(\K)$ contains a central involution that does not fix any of the vertices (see \cite[p. 255]{arp}); any such central involution $\alpha$ pairs up the vertices of $\K$ and naturally determines an equifibered covering $\K \rightarrow\K/\langle\alpha\rangle$, of $\K$ onto its quotient $\K/\langle\alpha\rangle$, satisfying the desired property with $l=2$.

Now returning to the general discussion, let $m,n\geq 2$ and $l$ be as above. Define $N:=\{1,\ldots,n\}$, $M:=\{1,\ldots,m\}$ and $L:=\{1,\ldots,l\}$. We wish to describe coverings $n^\K \rightarrow m^\LL$ that can balance the effect of $\gamma$ as an $l:1$ mapping on the vertex sets, by a controlled change in the base parameters from $n$ to $m$, provided $m\leq n^l$. To this end, we may assume that the vertices of $\K$ and $\LL$ are labeled in such a way that 
\[ \gamma^{-1}(j) = L_j := \{(j-1)l+1,\ldots,(j-1)l+l\}   \qquad (\mbox{for }j\in V(\LL)). \]
Thus for each~$j$, the map $\gamma$ takes the vertices of $\K$ in $L_j$ to the vertex $j$ of $\LL$. By a slight abuse of notation, we then can write a vector $\varepsilon = (\varepsilon_1,\ldots,\varepsilon_{v(\K)})$ in $N^{v(\K)}=(N^l)^{v(\LL)}$ in the form $\varepsilon = ({\hat{\varepsilon}}_1,\ldots,{\hat{\varepsilon}}_{v(\LL)})$, where
\[ {\hat{\varepsilon}}_j := (\varepsilon_{(j-1)l+1},\ldots,\varepsilon_{(j-1)l+l}) \] 
lies in $N^l$ for each $j=1,\ldots,v(\LL)$. 

Now suppose that, in addition to $\gamma$, we also have a surjective mapping $g\!: N^l \rightarrow M$ (and hence $m\leq n^l)$. Then $\gamma$ and $g$ determine a mapping
\begin{equation}
\label{gabet}
\begin{array}{rccc}
\pi^{\gamma,g}\!: & n^\K & \rightarrow & m^\LL\\[.02in]
& F(\varepsilon) & \rightarrow & (F\gamma)(\varepsilon_g),
\end{array}
\end{equation}
where again $F$ denotes a face of $\K$ and $\varepsilon$ a vector in $N^{v(\K)}$, and  
\[ \varepsilon_g := ({\hat{\varepsilon}}_1g,\ldots,{\hat{\varepsilon}}_{v(\LL)}g) \]
is the vector in $M^{v(\LL)}$ given by the images under $g$ of the components of $\varepsilon$ in its representation as $({\hat{\varepsilon}}_1,\ldots,{\hat{\varepsilon}}_{v(\LL)})$. We must prove that $\pi:=\pi^{\gamma,g}$ is a covering. 

First we must show that $\pi$ is well-defined. Suppose we have $F(\varepsilon)=F'(\varepsilon')$ in $n^\K$, where $\varepsilon = (\varepsilon_1,\ldots,\varepsilon_{v(\K)})$ and $\varepsilon' = (\varepsilon'_1,\ldots,\varepsilon'_{v(\K)})$ belong to $N^{v(\K)}$ and $F,F'$ are faces of $\K$. Then, as in the proof of the previous theorem, $F=F'$, $F\gamma=F'\gamma$, and $\varepsilon_{i}=\varepsilon'_{i}$ for $i\not\in V(F)=V(F')$. Now bear in mind that $\gamma$ is a covering. Hence, if $i\in V(F)$ then $(i)\gamma\in V(F\gamma)$; or equivalently, if $j\not\in V(F\gamma)$ then $V(F)\cap L_j =\emptyset$. It follows that, if $j\not\in V(F\gamma)$, then $\varepsilon_{i}=\varepsilon'_{i}$ for every $i$ in $L_j$, and therefore ${\hat{\varepsilon}}_j = {\hat{\varepsilon'}}_{\!\!j}$ and ${\hat{\varepsilon}}_jg = {\hat{\varepsilon'}}_{\!\!j}g$. Hence $\varepsilon_g$ and $\varepsilon'_g$ agree on every component represented by vertices of $\LL$ outside of $F\gamma=F'\gamma$. As the remaining components are allowed to take any value in $M$, we conclude that $(F\gamma)(\varepsilon_g)=(F'\gamma)(\varepsilon'_g)$. Thus $\pi$ is well-defined.

It is straightforward to verify that $\pi$ is a rank-preserving surjective homomorphism. To show that $\pi$ is also weakly adjacency preserving, let
\[ \Phi(\varepsilon) := \{\emptyset, \varepsilon, F_{0}(\varepsilon),\ldots,F_{k}(\varepsilon)\},\;\;\,  
\Phi'(\varepsilon') := \{\emptyset, \varepsilon', F'_{0}(\varepsilon'),\ldots,F'_{k}(\varepsilon')\}\] 
be $r$-adjacent flags of $n^\K$, where 
\[ \Phi:=\{F_{-1},F_0,\ldots,F_k\},\;\;\, \Phi':=\{F'_{-1},F'_0,\ldots,F'_k\}\] 
are flags of $n^\K$ and $\varepsilon,\varepsilon'$ lie in $N^{v(\K)}$. Again two possibilities arise. 
First, if $r>0$ then $\varepsilon=\varepsilon'$ and $\Phi,\Phi'$ are $(r-1)$-adjacent in $\K$. Hence $\varepsilon_g =\varepsilon'_g$ and $\Phi\gamma,\Phi'\gamma$ are $(r-1)$-adjacent in $\LL$. It follows that the two image flags under $\pi$,
\[\begin{array}{ccl}
(\Phi(\varepsilon))\pi \!\!&\!\!=\!\!&\!\!
\{\emptyset, \varepsilon_g,(F_{0}\gamma)(\varepsilon_g),\ldots,(F_{k}\gamma)(\varepsilon_g)\},\\[.06in]
(\Phi'(\varepsilon'))\pi \!\!&\!\!=\!\!&\!\!
\{\emptyset, \varepsilon'_g,(F'_{0}\gamma)(\varepsilon'_g),\ldots,(F'_{k}\gamma)(\varepsilon'_g)\}, 
\end{array}\]
are also $r$-adjacent. Now, if $r=0$ then $\Phi=\Phi'$ (but $\varepsilon\neq\varepsilon'$); in fact, $V(F_{s})=V(F'_{s})$ and hence $F_{s}=F'_{s}$ for each $s\geq 0$. When $s=0$ this gives $F_{0}(\varepsilon)=F'_{0}(\varepsilon')=F_{0}(\varepsilon')$ (since $r\neq 1$), and therefore $\varepsilon_{i}=\varepsilon'_{i}$ for each vertex $i$ of $\K$ distinct from $i_{0}:=F_0$; hence $\varepsilon$ and $\varepsilon'$ only differ in the position indexed by $i_{0}$. This already implies that $(F_{s}\gamma)(\varepsilon_g)=(F'_{s}\gamma)(\varepsilon'_g)$ for all $s\geq 0$, and hence that $(\Phi(\varepsilon))\pi$ and $(\Phi'(\varepsilon'))\pi$ are weakly $0$-adjacent flags of $m^\LL$. Thus $\pi$ is a weak covering. 

Moreover, since $\varepsilon_{i}=\varepsilon'_{i}$ if and only if $i\neq i_0$, we also know that ${\hat{\varepsilon}}_{\!j} ={\hat{\varepsilon'}}_{\!\!j}$ if and only if $j\neq j_{0}:=(i_0)\gamma$. Hence the two vertices 
\[ \varepsilon_g := ({\hat{\varepsilon}}_1g,\ldots,{\hat{\varepsilon}}_{v(\LL)}g), \;\; 
\varepsilon'_g := ({\hat{\varepsilon'}}_{\!\!1}g,\ldots,{\hat{\varepsilon'}}_{\!\!v(\LL)}g) \] 
of $m^\LL$ lying in $(\Phi(\varepsilon))\pi$ and $(\Phi'(\varepsilon'))\pi$, respectively, either coincide or differ in a single position, indexed by $j_{0}$; the former occurs precisely when ${\hat{\varepsilon}}_{\!j_0}g={\hat{\varepsilon'}}_{\!\!j_0}g$. Since ${\hat{\varepsilon}}_{\!j_0}$ can take any value in $N^l$, the mapping $\pi$ is a covering if and only if $g$ is a bijection.

Finally, suppose $g$ is a bijection, so in particular $m=n^l$. Then $n^\K$ and $m^\LL$ must have the same number of vertices,
\[ n^{v(\K)} = n^{l\cdot v(\LL)} = m^{v(\LL)}, \]
and hence $\pi$ must be a covering that is one-to-one on the vertices.
\medskip

In summary, we have established the following theorem. 

\begin{theorem}
\label{covindnew}
Let $\K$ and $\LL$ be finite vertex-describable incidence complexes of rank $k$, let $\gamma: \K \rightarrow \LL$ be a covering, and let $m,n\geq 2$ and $l\geq 1$. Suppose that $\gamma$ is equifibered with vertex fibers of cardinality $l$, and that $g\!: \{1,\ldots,n\}^l \rightarrow \{1,\ldots,m\}$ is a surjective mapping (and hence $m\leq n^l$). Then $\gamma$ and $g$ induce a weak covering $\pi^{\gamma,g}\!:\, n^\K \rightarrow m^\LL$ between the power complexes $n^\K$ and $m^\LL$. Moreover, $\pi^{\gamma,g}$ is a covering if and only if $g$ is a bijection (and $m=n^l$); in this case $\pi^{\gamma,g}$ is one-to-one on the vertices.
\end{theorem}

As an example consider finite regular polygons $\K = \{2p\}$ and $\LL = \{p\}$, with $2p$ or $p$ vertices, respectively, for some $p\geq 2$. The central symmetry of $\K$ gives an obvious equifibered covering $\gamma\!:\K\rightarrow\LL$ between $\K$ and $\LL$ with fibers of size $l=2$. Now choose $m=n^2$ and pick any bijection $g\!:\{1,\ldots,n\}^2 \rightarrow\{1,\ldots,n^2\}$. Then 
\[ \pi^{\gamma,g}\!:\, n^{\{2p\}} \rightarrow (n^2)^{\{p\}} \] 
is a covering. Either complex has $n^{2p}$ vertices, and $\pi^{\gamma,g}$ is one-to-one on the vertices. For example, when $n=2$ we obtain a covering
\[ \pi^{\gamma,g}\!:\, 2^{\{2p\}} \rightarrow 4^{\{p\}}. \] 
Here $2^{\{2p\}}$ is Coxeter's regular map $\{4,2p \!\mid\!4^{\lfloor p\rfloor -1}\}$ described in Section~\ref{app}.

\bibliographystyle{amsplain}

\end{document}